\documentclass[12pt]{article}

\usepackage{blindtext}
\usepackage{amssymb,amsfonts,amsmath,mathdots,amsthm}
\usepackage{titlesec,titletoc}
\usepackage{newtxmath,newtxtext}
\usepackage[T1]{fontenc}
\usepackage{tikz,graphicx}
\usepackage{cite}
\usepackage[hidelinks]{hyperref}
\usepackage[outline]{contour}
\usetikzlibrary{arrows,shapes.misc}
\usetikzlibrary{shapes.misc}

\usepackage{enumitem}

\numberwithin{equation}{section}

\theoremstyle{definition}
\newtheorem{defn}{Definition}[section]

\theoremstyle{definition}
\newtheorem{prop}[defn]{Proposition}

\newtheorem{theor}[defn]{Theorem}

\newtheorem{obse}[defn]{Observation}

\theoremstyle{remark}

\title{\textbf{Open, convex, unbounded sets in normed spaces.}}
\author{D. Moshonas, V. Nestoridis, A. Terezakis}

\begin{document}

\maketitle{}

\fontsize{11}{11}\selectfont

\begin{center}
\vspace*{\stretch{1}}
\textbf{Abstract}
\vspace*{\stretch{1}}
\end{center}
Let X be a normed linear space. We examine if every open, convex and unbounded subset of X is equal to the union of a family of open straight half lines. The answer is affirmative if and only if X is finite dimensional.


\section{Introduction} 
The inequality of Kolmogorov (\cite{Kolmogorov} ) implies that if $f$ is defined on a half-line, say $\left( 0, +\infty \right)$, and the derivatives $f^{(a)},f^{(b)}$, $a<b$, are bounded, then each intermediate derivative $f^{(l)}$, $a<l<b$ is also bounded. This fact was used in \cite{Nestoridis}  to prove that if a domain $\Omega \subset \mathbb{C}$ is the union of open half-lines then $H_F^{\infty}(\Omega)=H_{\tilde{F}}^{\infty}(\Omega)$. Here $F$ is a subset of $\{0,1,2,\dots\}$ and the space $H_F^{\infty}(\Omega)$ is the set of holomorphic in $\Omega$ functions f such that every derivative $f^{(l)}$, $l\in F$ is bounded in $\Omega$. And $\tilde{F}=\{l\in \{0,1,2,\dots\}: \exists b\in F \text{ with } minF\leq l \leq b\}$.  Furthermore, in \cite{Nestoridis} they were interested to unbounded convex domain in the plane $\mathbb{C}$, for which $H_F^{\infty}(\Omega)=H_{\tilde{F}}^{\infty}(\Omega)$. Thus, they stated and proved the following

\begin{theor}\label{0} Every open, convex, and unbounded subset of $\mathbb{R}^2$ contains an open straight half-line, and in fact it is equal to a union of such half-lines. Furthermore, we can obtain that these half-lines are parallel to each other.
\end{theor}

The proof of Theorem \ref{0} in \cite{Nestoridis} can easily transferred  to any finite dimensional space $\mathbb{R}^n$. In the present paper we show that in any infinite dimensional normed vector space there is an open convex and unbounded subset $\Omega$ which does not contain any straight half-line. Therefore, Theorem \ref{0} remains valid for a normed linear space $X$ instead of $\mathbb{R}^2$, if and only if, $X$ is finite dimensional.
In our proof, the fact that $\Omega$ contains a half-line and the fact that $\Omega$ is a union of such half-lines are proven to be equivalent for every normed space. We also treat first the separable case and then the general one. In the separable case we use the existence of a biorthogonal system (\cite{Lindenstrauss} ).


\section{Basic properties}

\begin{prop}\label{oneforall} Let $X$ be a normed linear space and let $\Omega$ be an open, convex and unbounded subset of $X$. Then the set $\Omega$ is a union of open half-lines if and only if $\Omega$ contains a half-line. We can also choose the half-lines to be parallel to each other.
\end{prop}

\begin{proof}
\begin{description}

\item{$(\Longrightarrow)$} Trivial.

\item{$(\Longleftarrow)$} Let $\Omega$ be an open, convex, unbounded subset of $X$, which contains a half-line. Without loss of generality we suppose, that $0\in \Omega$ and that the half-line starts from zero. In other words, there is $u_0\in X$, with $\|u_0\|=1$ such that $\{tu_0, t>0\}\subseteq \Omega$. 

\item{\textbf{Assertion.}} If $z\in \Omega$ then $\{z+tu_0, t>0\} \subseteq \Omega$.\\
If the assertion is true, we have

\begin{equation} \label{<} \bigcup_{z\in \Omega}\{z+tu_0, t>0\}\subseteq \Omega. \end{equation}

On the other hand, for an arbitrary $z_0\in \Omega$ there is $\delta>0$ such that $B(z_0,\delta)\subseteq \Omega$ hence, $z_0-(\delta/2)u_0\in \Omega$ and obviously $z_0\in \{(z_0-(\delta/2)u_0)+tu_0, t>0\}$, thus

\begin{equation*} z_0\in\bigcup_{z\in \Omega}\{z+tu_0, t>0\}. \end{equation*} 

Since $z_0$ was an arbitrary element of $\Omega$, we have 

\begin{equation} \label{>} \Omega \subseteq \bigcup_{z\in \Omega}\{z+tu_0, t>0\}. \end{equation}

Now the relations (\ref{<}) and (\ref{>}) provide the result.

\item{Proof of the assertion.}
Let $z_0\in \Omega$ and assume that there exists $t_0>0$ such that $a_0=z_0+t_0u_0\notin \Omega$. Since the set $\Omega$ is open, there is $\delta>0$ such that $B(z,\delta)\subseteq \Omega$. First we will show that there exist $t>0$ and $\lambda=\lambda(t)<0$ such that
\begin{equation}
\label{central} \xi_{t,\lambda}=(1-\lambda)a_0+\lambda(tu_0)\in B(z_0, \delta) \text{; that is } \|(1-\lambda)a_0+\lambda(tu_0)- z_0\|<\delta.
\end{equation}

We know that

\begin{equation*}
\frac{tu_0-a_0}{\|tu_0-a_0\|} \longrightarrow \frac{u_0}{\|u_0\|}=u_0, \text{ as } t\longrightarrow +\infty,
\end{equation*}

and thus, there exists $t_1>0$ such that

\begin{equation*}\left\Vert\frac{tu_0-a_0}{\|tu_0-a_0\|}-u_0\right\Vert<\frac{\delta}{2t_0} \text{ for all } t\geq t_1.
\end{equation*}

Now we rewrite the point $\xi_{t, \lambda}$, as:

\begin{equation*}\xi_{t,\lambda}=(1-\lambda)a_0 +\lambda(tu_0)=a_0+\lambda\left(tu_0-a_0\right)
\end{equation*}

For $t=t_1>0$ and $\lambda=-t_0\left(\|t_1u_0-a_0\|\right)^{-1} <0 $ we have

\begin{equation*}
\left\Vert\xi_{t,\lambda}-z_0 \right\Vert = \left\Vert a_0+\left(-t_0\right)\frac{t_1u_0-a_0}{\|t_1u_0-a_0\|}-z_0\right\Vert=\left\Vert t_0u_0+\left(-t_0\right)\frac{t_1u_0-a_0}{\|t_1u_0-a_0\|}\right\Vert=
\end{equation*}

\begin{equation*}=t_0\left\Vert \frac{t_1u_0-a_0}{\|t_1u_0-a_0\|}-u_0\right\Vert<t_0\frac{\delta}{2t_0}=\frac{\delta}{2}<\delta
\end{equation*}

Therefore, we found $t>0$ and $\lambda<0$ such that $\xi_{t,\lambda}\in B(z_0,\delta)\subseteq \Omega$. We can now write the $a_0$ as convex combination of elements of $\Omega$.

\begin{equation*}
\xi_{t,\lambda}=(1-\lambda)a_0+\lambda(tu_0) \text{ therefore, } a_0=\frac{1}{1-\lambda}\xi_{t,\lambda}+\frac{-\lambda}{1-\lambda}(tu_0)
\end{equation*}

and $\lambda<0$ hence,
\begin{equation*} 0<\frac{1}{1-\lambda}, \ \frac{-\lambda}{1-\lambda}<1.
\end{equation*}

Since the $\Omega$ is a convex set, we conclude that $a_0\in \Omega$, which is a contradiction.

\end{description}
\end{proof}

\begin{obse}\label{indip} Let $\Omega$ be a set as previously mentioned and $A\in \Omega$, if\\
$L(A,\Omega)=\{u\in X \ , \ \|u\|=1 \text{ and } \{A+tu, t>0\}\subseteq \Omega\}$.\\
Then the set $L(A,\Omega)$ is independent of $A$.
\end{obse}

\begin{proof} It is evident from the proof of the Proposition \ref{oneforall}.
\end{proof}

\begin{theor}\label{ithusd} Let $\Omega$ be an open, convex and unbounded subset of a normed space. Then the following are equivalent.
\begin{description}
\item{$(i)$} There is a sequence $w_n$ in $\Omega$ such that $\|w_n\|\longrightarrow +\infty$, and the sequence $w_n/\|w_n\|$ has a limit point.
\item{$(ii)$} The set $\Omega$ contains a half-line.
\item{$(iii)$} The set $\Omega$ is a union of parallel half-lines.
\item{$(iv)$} The set $\Omega$ is a union of half-lines.
\end{description}

\end{theor}

\begin{proof}
The implication $(ii)\Rightarrow (iii)$ has already been proven and the implications $(iii)\Rightarrow (iv) \Rightarrow (i)$ are obvious; thus, it remains to show that $(i)\Rightarrow (ii)$. Without loss of generality we assume that $0\in \Omega$. Due to the fact that $w_n/\|w_n\|$ has a limit point there exists a subsequence $k_n$ such that $u_n=w_{k_n}/\|w_{k_n}\|\longrightarrow u_0$. We fix now an arbitrary positive number $t$, and let $b_n=2t u_n$. Since $\|w_n\|\longrightarrow +\infty$, there exists $n_1\in \mathbb{N} \text{ such that } 2t<\|w_{k_n}\|$. This implies that $2t\|w_{k_n}\|^{-1}<1 \text{ and } b_n\in \Omega \text{, for all } n\geq n_1$. Also we know that $2t u_0-b_n\longrightarrow 0\in \Omega, \text{ as } n\longrightarrow +\infty$ and the set $\Omega$ is open; thus there is $n_2\in \mathbb{N}$ such that $2t u_0-b_n\in \Omega$, for all $n\geq n_2$. We choose $n\geq max\{n_1,n_2\}$ and we have $b_n, (2t u_0-bn)\in \Omega$. However, the set $\Omega$ is convex; hence the point $t u_0=b_n/2+(2t u_0-b_n)/2$ is in $\Omega$. Since $t$ was an arbitrary positive number, it follows that $\{t u_0, t >0\} \subseteq \Omega$, i.e. the set $\Omega$ contains a half-line.
\end{proof}

\begin{theor} \label{finite!} Let $X$ be a finite dimensional normed space. Then every open, convex unbounded subset of $X$, is a union of parallel half-lines.
\end{theor}

\begin{proof} Let $w_n$ be a sequence in $\Omega$, such that $\|w_n\|\longrightarrow +\infty$. The sequence $w_n/\|w_n\|$ belongs to $S(0,1)=\{v\in X : \|v\|=1\}$. Since the $X$ is finite dimensional space the sphere $S(0,1)$ is compact, and hence the sequence $w_n/\|w_n\|$ has a limit point. Now the Theorem \ref{ithusd} yields the result.
\end{proof}

\section{The separable case} 

We need the following well known fact.

\begin{prop}\label{CBS} \cite{Lindenstrauss} Every separable Banach space has a sequence $\left(x_n\right)_{n\in \mathbb{N}}$ in X and a sequence $\left(x^*_n\right)_{n\in \mathbb{N}}$ in $X^*$, such that

\begin{description}
\item{$(i)$} $x^*_i(x_j)=\delta_{ij}$.
\item{$(ii)$} If $x\in X$ and $x^*_n(x)=0, \text{ for all } n$, then $x=0$.
\end{description}
\end{prop}

\begin{prop}\label{main} Let $X$ be a separable Banach space and $(x_n)_{n\in \mathbb{N}}$, $(x_n^*)_{n\in \mathbb{N}}$ be sequences as in Proposition \ref{CBS}. Let $\Omega=\{x\in X : |x^*_n(x)|<R_n\}$. Then there exists a choice of the sequence $R_n, \ n=1,2,\dots$ such that, the set $\Omega$ is open, convex, unbounded and does not contain any half-line.
\end{prop}

\begin{proof} We can choose any sequence $\epsilon_n$ such that $\epsilon_n\longrightarrow +\infty$ (for example $\epsilon_n=n$) and then we define $R_n=\epsilon_n\|x^*_n\|$.

\begin{description}

\item{$(i)$} Let $x_0\in \Omega$; since $\epsilon_n\longrightarrow +\infty$, there is $n_0\in \mathbb{N}$ such that $\epsilon_n-\|x_0\|>1$ for all $n>n_0$. Now we fix a $\delta$, such that

\begin{equation*}
0<\delta<min\{1,\frac{R_n-|x^*_n(x_0)|}{\|x^*_n\|}, \ n=1,\dots,n_0\}
\end{equation*}

Let $u\in X, \text{ with } \|u\|< \delta. \text{ For all } N \in \{1,\dots,n_0\}$ we have

\begin{equation*}\delta<min\{\frac{R_n-|x^*_n(x_0)|}{\|x^*_n\|}, \ n=1,\dots,n_0\} \leq \frac{R_N-|x^*_N(x_0)|}{\|x^*_N\|}.
\end{equation*}

This implies

\begin{equation}\label{finite} 
\delta\|x^*_N\|< R_N - |x^*_N(x_0)| \text{ and } |x^*_N(x_0)|+\delta\|x^*_N\|< R_N.
\end{equation}

Hence for all $n=1,\dots,n_0$, it holds

\begin{equation*} 
|x^*_n(x_0+u)| \leq |x^*_n(x_0)|+|x^*_n(u)| \leq |x^*_n(x_0)|+\|x^*_n\|\|x_0\|< |x^*_n(x_0)|+\delta \|x^*_n\|
\end{equation*}

and from the (\ref{finite}) we finally have $|x^*_n(x_0+u)|<R_n$, for all $n=1,\dots, n_0$.\\
Now let $n>n_0$ and $u\in X$ with $\|u\|<\delta$. Since $\delta<1<\epsilon_n-\|x_0\|$ it follows

\begin{equation*}
|x^*_n(u)| \leq \delta\|x^*_n\|<\|x^*_n\|\left(\epsilon_n-\|x_0\|\right) \leq R_n-|x^*_n(x_0)|,
\end{equation*}

\begin{equation}\label{infinite} 
\text{and } |x^*_n(x_0)|+|x^*_n(u)|<R_n
\end{equation}

Therefore, $|x^*_n(x_0+u)|\leq |x^*_n(x_0)|+|x^*_n(u)|<R_n$, for all $n>n_0$. We have shown that for all $u\in X$, with $\|u\|<\delta$, $x_0+u\in \Omega$; that is, $B(x_0,\delta)\subseteq \Omega$. Since the element $x_0$ was arbitrary, the set $\Omega$ is open. 

\item{$(ii)$} Clearly the set $\Omega$ is convex.

\item{$(iii)$} We define the sequence
\begin{equation*} a_n=\dfrac{R_n}{2|x^*_n(x_n)|}x_n.
\end{equation*}
It is obvious that, $a_n\in \Omega \text{, for all }  n$, because

\[
   |x^*_m(a_n)| =
  \begin{cases}
    R_n/2 & ,\text{if $n=m$} \\
    0 & ,\text{if $n\neq m$}.
  \end{cases}
\]

On the other hand

\begin{equation*} \|a_n\|=\dfrac{R_n}{2|x^*_n(x_n)|}\|x_n\|\geq \dfrac{\epsilon_n\|x^*_n\|}{2\|x^*_n\|\|x_n\|}\|x_n\|=\dfrac{\epsilon_n}{2}\longrightarrow +\infty,
\end{equation*}

 as $n\longrightarrow +\infty$. Hence the $\Omega$ is unbounded.
 
\item{$(iv)$} Obviously the set $\Omega$ contains zero. Since $\Omega$ is open, convex and unbounded, from the Observation \ref{indip} it suffices to prove that there is no half-line which starts from 0 and is contained in $\Omega$. Assume that there is  $u_0\in X$ with $u_0\neq 0$ such that $\{tu_0, t>0\} \subseteq \Omega$.\\
Fix $N\in \mathbb{N}$; now $|x_N^*(u_0)|=|x_N^*(tu_0)|t^{-1}<R_Nt^{-1}\longrightarrow 0$; as $t\longrightarrow +\infty$; thus $x_n^*(u_0)=0, \text{ for all } n \in \mathbb{N}$. According to Proposition \ref{CBS} it follows that $u_0=0$, which is a contradiction. Hence there is no half-line in $\Omega$.
\end{description}
\end{proof}

\begin{prop}\label{sub} Let $X$ be a Banach space, containing a subset $\Omega$ which is open, convex, unbounded and does not contain any half-line. Let $Y$ be a dense linear subspace of $X$. Then the set $\Omega_Y=\Omega \cap Y$ is also open in $Y$, convex, unbounded and does not contain any half-line.
\end{prop}

\begin{proof}
\begin{description}

\item{$(i)$} Clearly the set $\Omega_Y$ is open and convex.

\item{$(ii)$} The set $\Omega_Y$ is unbounded. Indeed let $M>0$, since $\Omega$ is unbounded there is $a$ in $\Omega$ with $\|a\|>M+1$. However, the set $\Omega$ is open; thus, there exists $0<\delta<1/2$ such that $B(a,\delta)\subseteq \Omega$. Since $Y$ is dense subspace of $X$, it follows $Y\cap B(a,\delta)\neq \varnothing$; hence, there is $b\in Y\cap B(a,\delta)\subseteq \Omega_Y$. From the triangle inequality it is obvious that $\|b\|> M$.

\item{$(iii)$} The set $\Omega_Y$ does not contain any half-line, because, if there is a half-line which is contained in $\Omega_Y$ then the same half-line is contained also in $\Omega$, a contradiction.
\end{description}
\end{proof}

\begin{theor} Every separable normed space contains an open, convex, unbounded set which does not contain any half-line.
\end{theor}

\begin{proof} The completion of every separable normed linear space is a separable Banach space. Now the result follows from Proposition \ref{sub}.
\end{proof}

\section{The general case} 

\begin{prop}\label{uncount} Every infinite dimensional Banach space $X$ contains an open, convex, unbounded set which does not contains any half-line.
\end{prop}

\begin{proof} Let $Y$ be a separable Banach subspace of $X$. We already know from Proposition \ref{main} that there is a subset $\Omega_Y$ of $Y$, which is open in $Y$, convex, unbounded and does not contain any half-line. We define the set $\Omega = \Omega_Y + B(0,1)=\{w+b \ | w\in \Omega_Y\ \text{ and } b\in B(0,1) \}$ and we will prove that the set $\Omega$ is open in $X$, convex, unbounded and does not contain any half-line.
\begin{description}
\item{$(i)$} Since

 \begin{equation*} \Omega =\bigcup_{w\in \Omega_Y} B(w,1) \text{ the set $\Omega$ is open.} \end{equation*}

\item{$(ii)$} Let $x_1=w_1+b_1,x_2=w_2+b_2\in \Omega$, where $w_1,w_2\in \Omega_Y, \ b_1,b_2\in B(0,1)$ and $0\leq \lambda,\mu\leq 1$ with $\lambda+\mu=1$. Since $B(0,1)$ and $\Omega_Y$ are both convex, it follows that

\begin{equation*} \lambda x_1+\mu x_2=\left(\lambda w_1+\mu w_2\right)+\left(\lambda b_1+\mu b_2\right)\in \Omega. \end{equation*}

Thus, $\Omega$ is convex.

\item{$(iii)$} The set $\Omega_Y$ is unbounded and $\Omega$ contains it. Therefore, $\Omega$ is also unbounded.

\item{$(iv)$} We recall that $0\in \Omega_Y\subseteq \Omega$. Since the $\Omega$ is open, convex and unbounded, from the Observation \ref{indip} it suffices to prove that $\Omega$ does not contain any half-line which starts from the zero. Assume that there is $u_0\in X$ with $\|u_0\|=1$, such that $\{tu_0\ | \ t>0\}\subseteq \Omega$. Since $tu_0\in \Omega, \text{ for all } t>0$ we have that

\begin{equation*}tu_0=w_t+b_t,\end{equation*}

where $w_t \in \Omega_Y$ and $b_t \in B(0,1)$.

For $t=n$ we can observe the following

\begin{itemize}

\item{$(i)$} $w_n\in \Omega_Y, \text{ for all } n$

\item{$(ii)$} $\|w_n\|=\|nu_0-b_n\|\geq n\|v_0\|-1\longrightarrow +\infty$, as $n\longrightarrow +\infty$ \\
$\text{ hence } \|w_n\|\longrightarrow +\infty, \text{ as } n\longrightarrow +\infty$

\item{$(iii)$} Finally

\begin{equation*} \dfrac{w_n}{\|w_n\|}=\dfrac{nu_0-b_n}{\|w_n\|}=\dfrac{nu_0}{\|n u_0 - b_n\|}+\dfrac{b_n}{\|w_n\|}.
\end{equation*}

Now it is obvious that

\begin{equation*}\dfrac{nu_0}{\|n u_0- b_n\|}=\dfrac{u_0}{\| u_0-b_n/n\|}\longrightarrow \dfrac{u_0}{\|u_0\|}=u_0
\end{equation*}

and

\begin{equation*}
 \dfrac{b_n}{\|w_n\|} \longrightarrow 0, \text{ as } n\longrightarrow +\infty.
\end{equation*}

Hence $w_n/\|w_n\|\longrightarrow u_0$, as $n\longrightarrow +\infty$ and since $Y$ is a Banach space, it follows that $u_0\in Y$. In particular the sequence $w_n/\|w_n\|$ has a limit point in $Y$.

\end{itemize}
By the proven facts $(i), (ii), (iii)$ and according to Theorem \ref{ithusd} we deduce that the open in $Y$ set $\Omega_Y$ contains a half-line, which is a contradiction.
\end{description}

\end{proof}

\begin{theor} \label{nonsepar} Every infinite dimensional normed linear space contains an open, convex, unbounded set which does not contain any half-line.
\end{theor}

\begin{proof} The proof follows by a combination of Propositions \ref{uncount} and \ref{sub}.
\end{proof}

We have proved the following theorem.

\begin{theor} A normed linear space is finite dimensional if and only if every open, convex unbounded subset of $X$ is a union of parallel half-lines.
\end{theor}

\begin{proof} The proof follows by a combination of Theorems \ref{finite!} and \ref{nonsepar}.
\end{proof}

\textbf{ Acknowledgment:} We would like to thank D. Gatzouras, A. Giannopoulos and S. Mercourakis for helpful communications.


\section*{E-mails and addresses}
Dionysios Moschonas\\
National and Kapodistrian University of Athens, Department of Mathematics\\
Panepistimiopolis, 157 84, Athens, Greece\\
e-mail: dmoschon@math.uoa.gr\\
\text{}\\
Vassili Nestoridis\\
National and Kapodistrian University of Athens, Department of Mathematics\\
Panepistimiopolis, 157 84, Athens, Greece\\
e-mail: vnestor@math.uoa.gr\\
\text{}\\
Aleksios Terezakis\\
National and Kapodistrian University of Athens, Department of Mathematics\\
Panepistimiopolis, 157 84, Athens, Greece\\
e-mail: aleksistere@gmail.com

\end{document}